\documentclass[a4paper,12pt, reqno]{amsart}
\usepackage{latexsym}
\usepackage{amssymb,amsfonts,amsmath,mathrsfs}
\newtheorem{theorem}{Theorem}[section]
\newtheorem{proposition}[theorem]{Proposition}
\newtheorem{lemma}[theorem]{Lemma}

\theoremstyle{definition}

\theoremstyle{remark}

\numberwithin{equation}{section}

\begin{document} 
\title 
[Central values of derivatives of Dirichlet $L$-functions] % Õoptional short form; for the running headÕ 
{Central values of derivatives of Dirichlet $L$-functions} % ÕMain titleÕ 
% ÕEach author has his or her own set of coordinates.Õ 
\author{H. M. Bui}
\address{Mathematical Institute, University of Oxford, Oxford, OX1 3LB UK}
\email{hung.bui@maths.ox.ac.uk}
\thanks{The first author is supported by an EPSRC Postdoctoral Fellowship.}

\author{M. B. Milinovich}
\address{Department of Mathematics, University of Mississippi, University, MS 38677 USA}
\email{mbmilino@olemiss.edu}
%\thanks{Author was supported in part by S. M. Gonek's NSF grant.}

% ÕAbstract comes before maketitle, as in the AMS classesÕ 
\begin{abstract} 
Let $\mathscr{C}_{q}^{+}$ be the set of even, primitive Dirichlet characters (mod $q$).  Using the mollifier method we show that $L(\tfrac{1}{2},\chi)\neq 0$ for at least half of the characters $\chi\in\mathscr{C}_{q}^{+}$. Here, $L(s,\chi)$ is the Dirichlet $L$-function associated to the character $\chi$.  This result was previously known to hold for a third of the $\chi\in\mathscr{C}_{q}^{+}$.  In addition, we show that almost all the characters $\chi\in\mathscr{C}_{q}^{+}$ satisfy $L^{(k)}(\tfrac{1}{2},\chi)\neq 0$ when $k$ and $q$ are large.  \end{abstract} 
\maketitle 
% ÕMain text starts here.Õ

\section{Introduction \& Statement of the Main Result}

An important topic in number theory is the behavior of families of $L$-functions and their derivatives inside the critical strip. In particular, questions concerning the order of vanishing of $L$-functions at special points on the critical line have received a great deal of attention. In the case of Dirichlet $L$-functions, it is widely believed that $L(\tfrac{1}{2},\chi)\neq 0$ for all primitive characters $\chi$. For quadratic characters $\chi$, this appears to have been first conjectured by Chowla; he states this as problem 3 in chapter 8 of \cite{Ch}.  

Though a proof of the non-vanishing of Dirichlet $L$-functions at the central point $s=1/2$ has remained elusive, there has been considerable progress in showing that $L(\tfrac{1}{2},\chi)$ is very often non-zero within various families of characters $\chi$. In \cite{IS}, Iwaniec and Sarnak show that at least 1/3 of Dirichlet $L$-functions in the family of primitive characters, to a large modulus $q$, do not vanish at the central point. This improves upon earlier work of Balasubramanian and Murty \cite{BM}. Soundararajan \cite{S} has shown that at least 7/8 of the central values in the family of quadratic Dirichlet L-functions are non-zero. More recently, Baier and Young \cite{BY} consider the family of Dirichlet $L$-functions associated to cubic and sextic characters and show that infinitely many (though not a positive proportion) of these functions are not zero at the central point.

In \cite{MV}, Michel and VanderKam consider the behavior of the derivatives of completed Dirichlet $L$-functions, $\Lambda(s,\chi)$, at the central point. (See \textsection 3, below, for a definition.) In particular, they show that for $\varepsilon>0$ and $q$ sufficiently large depending on $\varepsilon$, the inequality 
\begin{equation}\label{PK}
{\sum_{\substack{\chi \ \!\!(\textrm{mod}\ \!q)\\ \Lambda^{(k)}({\scriptstyle{\frac{1}{2}}},\chi)\ne0}}{\!\!\!\!\!\!\!\!}}^{+} \ 1 \ \geq \ \Big( \ \!\!P_{k} - \varepsilon \ \!\! \Big) \cdot {\sum_{\chi \ \!\!(\textrm{mod}\ \!q)}{\!\!\!\!\!}}^{+} \ 1
\end{equation}
holds, where the proportion
$$ P_{k} = \frac{2}{3}-\frac{1}{36k^{2}}-\frac{c}{k^{4}}$$
for some absolute constant $c>0$. As $k$ tends to infinity, the proportion $P_{k}$ approaches two thirds. This is analogous to a result of Conrey \cite{C}, who shows that almost all of the zeros of the $k$-th derivative of the Riemann $\xi$-function are on the critical line, and to a result of Kowalski, Michel and VanderKam \cite{KMV} who show that almost half of the set $\big\{\Lambda^{(k)}(\tfrac{1}{2},f)\big\}$ is non-zero, where $f$ runs over the set of primitive Hecke eigenforms of weight 2 relative to $\Gamma_{0}(q)$. This last result is best possible because half of these forms are even and half are odd. However, unlike the results in \cite{C} and \cite{KMV}, the inequality in (\ref{PK}) is not best possible since it is expected that $P_{k}=1$ for every positive integer $k$.

In contrast to \cite{MV}, we study the behavior of the functions $L^{(k)}(s,\chi)$, the derivatives of Dirichlet $L$-functions, at $s=\tfrac{1}{2}$. When $k$ and $q$ are sufficiently large, we show that $L^{(k)}(\tfrac{1}{2},\chi)\neq 0$ for almost all of the even, primitive characters $\chi$. As is the case in \cite{C} and \cite{KMV}, our result is asymptotically best possible as $k$ tends to infinity.

\begin{theorem}\label{th2}
Let $k\in\mathbb{N}$. Then, for $\varepsilon > 0$ and $q$ sufficiently large \textup{(}depending on $\varepsilon$\textup{)}, we have
\begin{equation}\label{PstarK}
{\sum_{\substack{\chi \ \!\!(\!\!\!\!\!\!\mod\ \!\!q)\\ L^{(k)}({\scriptstyle{\frac{1}{2}}},\chi)\ne0}}{\!\!\!\!\!\!\!\!}}^{+} \ 1 \ \geq \ \Big( \ \!\!P_{k}^{*} - \varepsilon \ \!\! \Big) \cdot {\sum_{\chi \ \!\!(\!\!\!\!\!\!\mod\ \!\!q)}{\!\!\!\!\!\!}}^{+} \ 1,
\end{equation}
where the proportion
\begin{equation}\label{pro}
P_{k}^{\ast} = 1-\frac{1}{16k^{2}}-\frac{c}{k^{4}}
\end{equation}
for some absolute constant $c>0$. In particular, $P_1^*\geq.7544 , P_2^*\geq.9083,$ $P_3^*\geq.9642$, $P_4^*\geq.9853$, $P_5^*\geq.9935$,  and $P_{25}^*\geq.9999$.
\end{theorem}

Theorem \ref{th2} confirms a prediction of Conrey and Snaith which arises from the $L$-functions Ratios Conjectures (see \textsection 8.1 of \cite{CS}). Their heuristic is based upon studying the behavior of the mollified moments of the derivatives of the Riemann zeta-function in $t$-aspect which they conjecture should behave similarly to the mollified moments of the derivatives of Dirichlet $L$-functions at the central point in $q$-aspect. This is in agreement with the conjectures of Keating and Snaith \cite{KS3,KS4} that suggest that both of these families of $L$-functions, the Riemann zeta-function in $t$-aspect and Dirichlet $L$-functions in $q$-aspect, should have the same underlying ``unitary" symmetry and so their (mollified) moments should behave similarly. See \cite{CF} for a detailed discussion of these ideas. In particular, our Proposition \ref{prop2} is a $q$-analogue of a result of Conrey and Ghosh\footnote{See equation (7) of \cite{CG}.} who computed the mollified moments of the derivatives of the Riemann zeta-function on the critical line.  

We remark that Theorem \ref{th2} does not improve upon the main result of \cite{MV}. In fact, for $k\in\mathbb{N}$, the zeros of the functions $L^{(k)}(s,\chi)$ and $\Lambda^{(k)}(s,\chi)$ are expected to behave quite differently. To illustrate this point, let $\chi$ be a primitive character and assume that the Riemann Hypothesis (RH) holds for the function $L(s,\chi)$. Then all the non-trivial zeros of $L(s,\chi)$ and all the zeros of $\Lambda(s,\chi)$ lie on the critical line $\text{Re } \! s\!=\!\tfrac{1}{2}.$ In addition, $L(s,\chi)$ has an infinite number of trivial zeros on the negative real axis. Since both $L(s,\chi)$ and $\Lambda(s,\chi)$ are entire functions, this distinction has a profound effect on the distribution of the zeros of their derivatives. The reason for this is the following classical result from the theory of entire functions: {\it If $F(s)$ is an entire function, then the zeros of $F'(s)$ lie within the convex hull of the zeros of $F(s)$.} Under the RH for $L(s,\chi)$, this implies that all the zeros of $\Lambda^{(k)}(s,\chi)$ lie on the line $\text{Re } \! s\!=\!\tfrac{1}{2}$. In contrast, the zeros of $L^{(k)}(s,\chi)$ are forced to lie in the half-plane $\text{Re } \! s\!\leq\! \tfrac{1}{2}$ and it is very likely the case that none of these zeros lie on the critical line.\footnote{We can show that if $q$ is sufficiently large, then the only zeros of $L'(s,\chi)$ on the critical line are the multiple zeros of $L(s,\chi)$. However, it is believed that the zeros of $L(s,\chi)$ are simple.} In particular, it is reasonable to conjecture that $L^{(k)}(\tfrac{1}{2},\chi)\!\neq\! 0$ for all primitive characters $\chi$ and all $k\in\mathbb{N}$. However, if $\chi$ is an even, real-valued (i.e. quadratic), primitive character, then the functional equation for $L(s,\chi)$ states that $\Lambda(s,\chi)\!=\!\Lambda(1-s,\chi)$. It follows from this that $\Lambda^{(k)}(\tfrac{1}{2},\chi)\!=\!0$ whenever $k$ is odd. Thus, the analogous conjecture for $\Lambda^{(k)}(\tfrac{1}{2},\chi)$ fails for infinitely many characters $\chi$. 

\subsection{Notation \& Conventions}

We say a Dirichlet character $\chi$ (mod $q$) is even if $\chi(-1)=1$.  We let $\mathscr{C}_q$ denote the set of primitive characters (mod $q$) and let $\mathscr{C}_{q}^{+}$ denote the subset of characters in $\mathscr{C}_q$ which are even. We put $\varphi^{+}(q) = \frac{1}{2}\varphi^{*}(q)$ where
$$  \varphi^{*}(q) = \sum_{k|q} \varphi(k) \mu(\tfrac{q}{k}) = \big| \mathscr{C}_q\big| \ \!; $$
the proof of this appears in Lemma 4.1, below.  It is not difficult to show that $\big|\mathscr{C}_{q}^{+}\big|= \varphi^{+}(q)+O(1)$.  In addition, we write $\sum_{\chi \ \!\!(\textrm{mod}\ \!q)}^{+}$ to indicate that the summation is restricted to $\chi \in \mathscr{C}_{q}^{+}$ and we write $\sum_{a(\textrm{mod}\ q)}^{\textstyle{\star}}$ and $\sum_n^\star$ to indicate that the summation is restricted to the residues $a(\textrm{mod}\ q)$ which are coprime to $q$ and to $n$ which are relatively prime to $q$, respectively.

\section{The Mollified Moments of $L^{(k)}({\scriptstyle{\frac{1}{2}}},\chi)$}

As may be expected, we prove Theorem \ref{th2} by computing certain mollified first and second moments of $L^{(k)}(\tfrac{1}{2},\chi)$ over the characters $\chi \in \mathscr{C}_{q}^{+}$ and then we use Cauchy's inequality.   

To each character $\chi \in \mathscr{C}_{q}^{+}$ we associate the function 
\begin{equation} \label{mollifier}
M(\chi) = M(\chi, P, y) := \sum_{n\leq y} \frac{\mu(n)\chi(n)}{\sqrt{n}} P\Big(\frac{\log y/n}{\log y} \Big),
\end{equation}
where $P$ is an arbitrary polynomial satisfying the conditions $P(0)=0$ and $P(1)=1$. The purpose of the function $M(\chi)$ is to smooth out or ``mollify" the large values of $L^{(k)}(\tfrac{1}{2},\chi)$ as we average over the $\chi \in \mathscr{C}_{q}^{+}$.  If we let
\begin{equation}
S_1(k,q) = {\sum_{\chi(\text{mod }q)}\!\!\!\!}^+ \ L^{(k)}(\tfrac{1}{2},\chi) M(\chi)
\end{equation}
and 
\begin{equation}
S_2(k,q) = {\sum_{\chi(\text{mod }q)}\!\!\!\!}^+ \ \big|L^{(k)}(\tfrac{1}{2},\chi)\big|^2 \big|M(\chi)\big|^2,
\end{equation}
then Cauchy's inequality implies that
\begin{equation} \label{cauchyineq}
{\sum_{\substack{\chi(\text{mod }q) \\ L^{(k)}({\scriptstyle{\frac{1}{2}}},\chi)\neq 0}}\!\!\!\!\!\!\!}^+  \ 1 \ \geq \ \frac{\big|S_1(k,q)\big|^2}{S_2(k,q)}.
\end{equation}
Thus, we require a lower bound for $|S_1(k,q)\big|$ and an upper bound for $S_2(k,q)$. Such estimates are provided by the following propositions.

\begin{proposition} \label{prop1}
Let $k$ be a positive integer.  Then, for $y=q^\vartheta$ and $0 < \vartheta < 1$, we have
\begin{displaymath}
S_{1}(k,q)  = (-1)^k\varphi^+(q)P(1) \log^{k} q \  \Big(1+O\Big(\frac{1}{\log q}\Big)\Big),
\end{displaymath}
where the implied constant depends on $\vartheta$ and $k$.
\end{proposition}

\begin{proposition} \label{prop2}
Let $k$ be a positive integer and $\varepsilon>0$ be arbitrary. Then, for $y=q^\vartheta$ and $0 < \vartheta < \tfrac{1}{2}$, we have
\begin{eqnarray*}
S_{2}(k,q) =\mathcal{C}_{k}(\vartheta) \ \!\! \varphi^+(q) \log^{2k}q \ \!\Big(1+O\Big(\frac{1}{(\log q)^{1-\varepsilon}}\Big)\Big),
\end{eqnarray*}
where 
\begin{displaymath}
\mathcal{C}_{k}(\vartheta)= \frac{\vartheta^{-1}}{2k\!+\!1}\int_{0}^{1}\!P'(x)^2 \ dx+\frac{1}{2}+\frac{\vartheta k^2}{2k\!-\!1}\int_{0}^{1}\!P(x)^2 \ dx,
\end{displaymath}
and the implied constant depends on $\vartheta$, $\varepsilon$, and $k$
\end{proposition}

It is clear from (\ref{cauchyineq}) and the propositions that in order to prove Theorem \ref{th2} we need to choose the polynomial $P$, for each $k\geq 1$, which minimizes the constant $\mathcal{C}_{k}(\vartheta) $. This is done in \textsection 6. It turns out that except for a term which is exponentially small (as a function of $k$), the optimal choice of $P$ is independent of the choice of $\vartheta.$ This is not surprising, since similar phenomena have been observed when mollifying high derivatives of the Riemann zeta-function and the Riemann $\xi$-function on the critical line, and also when mollifying high derivatives of families of $L$-functions at the central point (see \cite{C,CG,KMV,MV}).

\section{Proof of Proposition \ref{prop1}}

In this section we establish Proposition \ref{prop1}. For $\chi\in\mathscr{C}_{q}^{+}$, the Dirichlet $L$-function, $L(s,\chi)$, associated to $\chi$ satisfies the functional equation
\begin{eqnarray}\label{fe}
\Lambda(s,\chi):=\bigg(\frac{q}{\pi}\bigg)^{s/2}\Gamma\bigg(\frac{s}{2}\bigg)L(s,\chi) =\varepsilon_\chi\Lambda(1-s,\overline\chi),
\end{eqnarray}
where $\varepsilon_\chi=\tau(\chi) q^{-1/2}$ and $\tau(\chi)$ is the Gauss sum
$$ \tau(\chi) = \sum_{a(\text{mod}\ \! q)} \chi(a) e\big(\tfrac{a}{q}\big) \ ; \quad e(x)=e^{2\pi i x}.$$
Note that $|\varepsilon_\chi|=1$ and, since $\chi$ is even, $\overline{\tau(\chi)}=\tau(\bar{\chi})$.

The result we require is implicit in \cite{MV} (see \textsection 3, page 135) where it is shown that\footnote{It follows from the functional equation for $\Lambda(s,\chi)$ that the quantity $\mathscr{L}(P_{k})$ in \textsection 3 of \cite{MV} is equal to $$ 2{\sum_{\chi \ \!\!(\textrm{mod}\ \!q)}}^{\!\!\!\!\!\!\! +} \ \Lambda^{(k)}(\tfrac{1}{2},\chi)M(\chi). $$}
\begin{equation}\label{mvmv}
{\sum_{\chi(\textrm{mod}\ \!q)}{\!\!\!\!\!\!}}^{+} \ \Lambda^{(k)}(\tfrac{1}{2},\chi)M(\chi) = \varphi^{+}(q) \ \! P(1) \ \!  \Gamma(\tfrac{1}{4}) \ \!  \hat{q}^{1/2} \log^{k}\hat{q} \ \Big(1+O\Big(\frac{1}{\log\hat{q}}\Big)\Big)
\end{equation}
for $k\in\mathbb{N}$ and $0 < \vartheta < 1$. Here $\hat{q}=\sqrt{q/\pi}$ and the implied constant depends on $\vartheta$. From (\ref{fe}), we see that 
\begin{equation} \label{fe2}
L(s,\chi) = H_{q}(s) \Lambda(s,\chi), \quad \text{ where} \ H_{q}(s) = \frac{ \hat{q}^{-s}}{\Gamma(\tfrac{s}{2})}.
\end{equation}
A straight-foward calculation shows that
\begin{equation} \label{fe3}
H_{q}^{(k)}(\tfrac{1}{2}) = (-1)^{k} \frac{\hat{q}^{-1/2}}{\Gamma(\tfrac{1}{4})} \log^{k}\hat{q}  \  \Big(1+O_{k}\Big(\frac{1}{\log\hat{q}}\Big)\Big)
\end{equation}
for each $k\in\mathbb{N}$. Now, combining (\ref{mvmv}), (\ref{fe2}), (\ref{fe3}) and using the Leibniz formula for differentiation, it follows that
\begin{equation*}
\begin{split}
{\sum_{\chi(\textrm{mod}\ \!q)}{\!\!\!\!\!\!}}^{+} \ L^{(k)}(\tfrac{1}{2},\chi)M(\chi) &=  \sum_{\ell=0}^{k}\binom{k}{\ell} {\sum_{\chi(\textrm{mod}\ \!q)}{\!\!\!\!\!\!}}^{+} \ H_{q}^{(\ell)}(\tfrac{1}{2}) \ \! \Lambda^{(k-\ell)}(\tfrac{1}{2},\chi) \ \! M(\chi)
\\
&= (-1)^{k} \sum_{\ell=0}^{k}\binom{k}{\ell} \varphi^{+}(q) \ \! P(1) \ \! \log^{k} \hat{q} \  \Big(1+O\Big(\frac{1}{\log\hat{q}}\Big)\Big)^{2}
\\
&= (-1)^{k} \ \! 2^{k} \ \! \varphi^{+}(q) \ \! P(1) \ \! \log^{k} \hat{q} \  \Big(1+O\Big(\frac{1}{\log\hat{q}}\Big)\Big),
\end{split}
\end{equation*}
where the implied constant depends on $\vartheta$ and $k$. Since $2 \log\hat{q} = \log q + O(1),$  we can conclude that
\begin{displaymath}
{\sum_{\chi(\textrm{mod}\ \!\!q)}{\!\!\!\!\!}}^{+} \ L^{(k)}(\tfrac{1}{2},\chi)M(\chi)= (-1)^{k}  \ \! \varphi^{+}(q) \ \! P(1)  \ \! \log^{k} q \  \Big(1+O\Big(\frac{1}{\log q}\Big)\Big).
\end{displaymath}
This establishes Proposition \ref{prop1}.

\section{Some Preliminary Results}

In this section we collect some preliminary results which we will use to establish Proposition \ref{prop2}. In what follows, $q$ is a large positive integer and $\alpha,\beta\in\mathbb{C}$ are taken to be small shifts satisfying $|\alpha|,|\beta| \leq 2(\log q)^{-1}$.  

Our first lemma concerns the orthogonality of primitive characters.

\begin{lemma}
For $(mn,q)=1$ we have
$$
{\sum_{\chi(\emph{mod}\ q)}{\!\!\!\!\!\!}}^{+}\ \chi(m)\overline{\chi}(n)=\frac{1}{2}\sum_{\substack{q=dr\\r|m\pm n}}\mu(d)\varphi(r),
$$
where the sums for the different signs $\pm$ are to be taken separately.
\end{lemma}
\begin{proof}
Let
\begin{displaymath}
f(h)={\sum_{\chi(\textrm{mod}\ h)}{\!\!\!\!\!\!}}^{\textstyle{*}}\ \ \chi(m)\overline\chi(n)
\end{displaymath}
where $\sum^{*}$ denotes summation over primitive characters $\chi$. Then for $(mn,q)=1$ we have
\begin{displaymath}
\sum_{h|q}f(h)=\sum_{\chi(\textrm{mod}\ q)}\chi(m)\overline\chi(n)=\left\{ \begin{array}{ll}
\varphi(q)\ \ \textrm{if $m\equiv n\ (\textrm{mod}\ q)$}\\
0 \qquad\textrm{otherwise.}
\end{array} \right.
\end{displaymath}
Using the M\"obius inversion we obtain
\begin{displaymath}
{\sum_{\chi(\textrm{mod}\ q)}{\!\!\!\!\!\!}}^{\textstyle{*}}\ \ \chi(m)\overline\chi(n)=f(q)=\sum_{\substack{h|q\\h|m-n}}\varphi(h)\mu(q/h).
\end{displaymath}
It follows from this identity that
$$\big| \mathscr{C}_q\big| ={\sum_{\chi(\textrm{mod}\ q)}{\!\!\!\!\!\!}}^{\textstyle{*}} \ 1 = \sum_{k|q} \varphi(k) \mu(\tfrac{q}{k}), $$
which justifies an above remark.  Our lemma now follows by noting that
\begin{displaymath}
{\sum_{\substack{\chi(\textrm{mod}\ q)\\\chi(-1)=1}}{\!\!\!\!\!\!}}^{\textstyle{*}}\ \ \chi(m)\overline\chi(n)={\sum_{\chi(\textrm{mod}\ q)}{\!\!\!\!\!\!}}^{\textstyle{*}}\ \ \bigg[\frac{1+\chi(-1)}{2}\bigg]\chi(m)\overline\chi(n).
\end{displaymath}
\end{proof}

\begin{lemma}
Let $G(s)$ be an even, entire function with rapid decay as $|s|\rightarrow\infty$ in any fixed vertical strip $A\leq\sigma\leq B$ and with $G(0)=1$. Let
\begin{equation}\label{1}
W_{\alpha,\beta}^{\pm}(x)=\frac{1}{2\pi i}\int_{(1)}G(s)H(s)g_{\alpha,\beta}^{\pm}(s)x^{-s}\frac{ds}{s},
\end{equation}
where
\begin{displaymath}
g_{\alpha,\beta}^{+}(s)=\frac{\Gamma(\frac{1/2+\alpha+s}{2})\Gamma(\frac{1/2+\beta+s}{2})}{\Gamma(\frac{1/2+\alpha}{2})\Gamma(\frac{1/2+\beta}{2})}, \quad \ g_{\alpha,\beta}^{-}(s)=\frac{\Gamma(\frac{1/2-\alpha+s}{2})\Gamma(\frac{1/2-\beta+s}{2})}{\Gamma(\frac{1/2+\alpha}{2})\Gamma(\frac{1/2+\beta}{2})},
\end{displaymath}
and
\begin{displaymath}
H(s)=\frac{(\frac{\alpha+\beta}{2})^2-s^2}{(\frac{\alpha+\beta}{2})^2}\qquad(\alpha+\beta\ne0).
\end{displaymath}
Then for $\chi_1,\chi_2\in\mathscr{C}_{q}^{+}$, $\alpha\neq-\beta$ we have that
\begin{eqnarray*}
L({\scriptstyle{\frac{1}{2}}}+\alpha,\chi_1)L({\scriptstyle{\frac{1}{2}}}+\beta,\chi_2)&=&\sum_{m,n}\frac{\chi_1(m)\chi_2(n)}{m^{1/2+\alpha}n^{1/2+\beta}}W_{\alpha,\beta}^{+}\bigg(\frac{\pi mn}{q}\bigg)\nonumber\\
&&\quad+\varepsilon_{\chi_1}\varepsilon_{\chi_2}\bigg(\frac{q}{\pi}\bigg)^{-\alpha-\beta}\sum_{m,n}\frac{\overline{\chi_1}(m)\overline{\chi_2}(n)}{m^{1/2-\alpha}n^{1/2-\beta}}W_{\alpha,\beta}^{-}\bigg(\frac{\pi mn}{q}\bigg).
\end{eqnarray*}
\end{lemma}

\noindent {\bf Some Remarks.}  
\begin{enumerate}  
\item An admissible choice of $G$ in the above lemma is $G(s)=\exp(s^2)$.  
\item The purpose of the function $H(s)$ in the above lemma is to cancel the poles of the functions $\zeta_q(1\pm(\alpha+\beta)+2s)$ at $s=\mp(\alpha+\beta)/2$ which appear in the next lemma. This substantially simplifies our later calculations. A similar effect has been previously observed by Conrey, Iwaniec and Soundararajan (see \textsection 3 of \cite{CIS}).
\end{enumerate}

\begin{proof}
Consider the integral
\begin{displaymath}
I_{\alpha,\beta}=\frac{1}{2\pi i}\int_{1-i\infty}^{1+i\infty}G(s)H(s)\frac{\Lambda(1/2+\alpha+s,\chi_1)\Lambda(1/2+\beta+s,\chi_2)}{\Gamma(\frac{1/2+\alpha}{2})\Gamma(\frac{1/2+\beta}{2})}\frac{ds}{s}.
\end{displaymath}
Shifting the line of integration to $\text{Re } \!{s}=-1$ and using Cauchy's theorem we obtain
\begin{displaymath}
I_{\alpha,\beta}=R_0+\frac{1}{2\pi i}\int_{-1-i\infty}^{-1+i\infty}G(s)H(s)\frac{\Lambda(1/2+\alpha+s,\chi_1)\Lambda(1/2+\beta+s,\chi_2)}{\Gamma(\frac{1/2+\alpha}{2})\Gamma(\frac{1/2+\beta}{2})}\frac{ds}{s},
\end{displaymath}
where $R_0$ is the term arising from the residue of the integrand at $s=0$. Evidently,  
\begin{displaymath}
R_0=\bigg(\frac{q}{\pi}\bigg)^{(1+\alpha+\beta)/2}L({\scriptstyle{\frac{1}{2}}}+\alpha,\chi_1)L({\scriptstyle{\frac{1}{2}}}+\beta,\chi_2).
\end{displaymath}
By making the change of variables $s$ to $-s$ and using (\ref{fe}), we have that
\begin{displaymath}
R_0=I_{\alpha,\beta}+\frac{1}{2\pi i}\int_{1-i\infty}^{1+i\infty}G(s)H(s)\frac{\Lambda(1/2-\alpha+s,\overline{\chi_1})\Lambda(1/2-\beta+s,\overline{\chi_2})}{\Gamma(\frac{1/2+\alpha}{2})\Gamma(\frac{1/2+\beta}{2})}\frac{ds}{s}.
\end{displaymath}
The lemma now follows by using (\ref{fe}) to express the $\Lambda$-functions as Dirichlet series and then integrating term-by-term.
\end{proof}

\begin{lemma}\label{21}
Let
\begin{displaymath}
S_{\alpha,\beta}^{+}(x)=\sum_{\substack{ n=1 \\ (n,q)=1}}^\infty \frac{W_{\alpha,\beta}^{+}(n^2/x)}{n^{1+\alpha+\beta}}\quad \text{ and } \quad S_{\alpha,\beta}^{-}(x)=\sum_{\substack{ n=1 \\ (n,q)=1}}^\infty \frac{W_{\alpha,\beta}^{-}(n^2/x)}{n^{1-\alpha-\beta}}.
\end{displaymath}
Then, for any $\varepsilon>0$ and $\alpha\neq-\beta$, we have that
\begin{displaymath}
S_{\alpha,\beta}^{+}(x)=\zeta_{q}(1+\alpha+\beta)+O(\tau(q)x^{-1/2+\varepsilon})
\end{displaymath}
and
$$S_{\alpha,\beta}^{-}(x)=g_{\alpha,\beta}^{-}(0)\zeta_{q}(1-\alpha-\beta)+O(\tau(q)x^{-1/2+\varepsilon}), $$
where $\tau(q)$ is the number of divisors of $q$ and the function $\zeta_q(s)$ is defined by
\begin{displaymath}
\zeta_q(s)=\zeta(s)\prod_{p|q}\bigg(1-\frac{1}{p^s}\bigg).
\end{displaymath}
\end{lemma}
\begin{proof}
From \eqref{1} we observe that
\begin{displaymath}
S_{\alpha,\beta}^{+}(x)=\frac{1}{2\pi i}\int_{(1)}G(s)H(s)g_{\alpha,\beta}^{+}(s)x^{s}\zeta_{q}(1+\alpha+\beta+2s)\frac{ds}{s}.
\end{displaymath}
We now move the line of integration to $\text{Re } \!{s}=-1/2+\varepsilon$, encountering only a simple pole of the integrand at $s=0$. We note that the simple pole of $\zeta_{q}(1+\alpha+\beta+2s)$ at $s=-(\alpha+\beta)/2$  is canceled by a zero $H(s)$. The residue of the integrand at $s=0$ is $\zeta_{q}(1+\alpha+\beta)$. Also, the integral along the new contour is trivially $\ll \tau(q)x^{-1/2+\varepsilon}$. This implies the first claim of the lemma.  The second claim can be proved in a similar manner. 
\end{proof}

\begin{lemma}
Assume $\alpha\neq-\beta$ and let
\begin{displaymath}
\mathscr{B}(m_1,n_1;\alpha,\beta)={\sum_{\chi(\emph{mod}\ \!q)}{\!\!\!\!\!}}^{+}\ L({\scriptstyle{\frac{1}{2}}}+\alpha,\chi)L({\scriptstyle{\frac{1}{2}}}+\beta,\overline{\chi})\chi(m_1)\overline{\chi}(n_1).
\end{displaymath}
Then for $(m_1,n_1)=1$ and $(m_1n_1,q)=1$ we have
\begin{eqnarray*}
\mathscr{B}(m_1,n_1;\alpha,\beta)&=&\frac{\varphi^{+}(q)}{\sqrt{m_1n_1}}\bigg(\frac{\zeta_{q}(1+\alpha+\beta)}{m_{1}^{\beta}n_{1}^{\alpha}}+\bigg(\frac{q}{\pi}\bigg)^{-\alpha-\beta}g_{\alpha,\beta}^{-}(0)\frac{\zeta_{q}(1-\alpha-\beta)}{m_{1}^{-\alpha}n_{1}^{-\beta}}\bigg)\nonumber\\
&&\qquad\qquad\qquad+O(\beta(m_1,n_1)+q^{1/2+\epsilon}),
\end{eqnarray*}
where $\beta(m_1,n_1)$ satisfies
\begin{displaymath}
\sum_{m_1,n_1\leq y}\frac{\beta(m_1,n_1)}{\sqrt{m_1n_1}}\ll yq^{1/2+\varepsilon}.
\end{displaymath}
\end{lemma}
\begin{proof}
With $\chi_1=\chi$, $\chi_2=\bar{\chi}$, Lemma 4.1 and Lemma 4.2 imply that
\begin{eqnarray}\label{3}
\mathscr{B}(m_1,n_1;\alpha,\beta)&=&\frac{1}{2}\sum_{q=dr}\mu(d)\varphi(r){\sum_{r|mm_1\pm nn_1}{\!\!\!\!\!\!\!\!}}^{\star}\ \ \frac{W_{\alpha,\beta}^{+}(\frac{\pi mn}{q})}{m^{1/2+\alpha}n^{1/2+\beta}}\nonumber\\
&&\quad+\frac{1}{2}\bigg(\frac{q}{\pi}\bigg)^{-\alpha-\beta}{\sum_{q=dr}\mu(d)\varphi(r)\sum_{r|mn_1\pm nm_1}{\!\!\!\!\!\!\!\!}}^{\star}\ \ \frac{W_{\alpha,\beta}^{-}(\frac{\pi mn}{q})}{m^{1/2-\alpha}n^{1/2-\beta}},
\end{eqnarray}
where $\sum^{\star}$ denotes summation over all $(mn,q)=1$. The main contribution to $\mathscr{B}(m_1,n_1;\alpha,\beta)$ comes from the diagonal terms $mm_1=nn_1$ and $mn_1=nm_1$ in the first and second sums on the right-hand side of (\ref{3}), respectively. For $(m_1,n_1)=1$, this contribution is
\begin{eqnarray*}
&&\varphi^{+}(q)\bigg( \ {\sum_{mm_1=nn_1}{\!\!\!\!\!\!\!}}^{\star}\ \ \frac{W_{\alpha,\beta}^{+}(\frac{\pi mn}{q})}{m^{1/2+\alpha}n^{1/2+\beta}}+\bigg(\frac{q}{\pi}\bigg)^{-\alpha-\beta}{\sum_{mn_1=nm_1}{\!\!\!\!\!\!\!}}^{\star}\ \ \frac{W_{\alpha,\beta}^{-}(\frac{\pi mn}{q})}{m^{1/2-\alpha}n^{1/2-\beta}}\bigg)\nonumber\\
&& \quad \quad \quad \quad  = \varphi^{+}(q)\bigg(\frac{S_{\alpha,\beta}^{+}(\frac{q}{\pi m_1n_1})}{n_{1}^{1/2+\alpha}m_{1}^{1/2+\beta}}+\bigg(\frac{q}{\pi}\bigg)^{-\alpha-\beta}\frac{S_{\alpha,\beta}^{-}(\frac{q}{\pi m_1n_1})}{m_{1}^{1/2-\alpha}n_{1}^{1/2-\beta}}\bigg),
\end{eqnarray*}
where $S_{\alpha,\beta}^{\pm}(x)$ are defined in Lemma 4.3. By Lemma 4.3 the above expression is equal to
\begin{equation*}\label{22}
\frac{\varphi^{+}(q)}{\sqrt{m_1n_1}}\bigg(\frac{\zeta_{q}(1+\alpha+\beta)}{m_{1}^{\beta}n_{1}^{\alpha}}+\bigg(\frac{q}{\pi}\bigg)^{-\alpha-\beta}g_{\alpha,\beta}^{-}(0)\frac{\zeta_{q}(1-\alpha-\beta)}{m_{1}^{-\alpha}n_{1}^{-\beta}}\bigg)+O(q^{1/2+\epsilon}).
\end{equation*}
All the other terms in \eqref{3} contribute at most
\begin{displaymath}
\beta(m_1,n_1)=\sum_{mm_1\ne nn_1}\frac{(mm_1\pm nn_1,q)}{\sqrt{mn}}\bigg|W_{\alpha,\beta}^{\pm}\bigg(\frac{\pi mn}{q}\bigg)\bigg|.
\end{displaymath}
Using the estimate $|W_{\alpha,\beta}^{\pm}(x)|\ll(1+x)^{-1}$ one can show that (see \cite{IS}, Section 4)
\begin{equation*}\label{4}
\sum_{m_1,n_1\leq y}\frac{\beta(m_1,n_1)}{\sqrt{m_1n_1}}\ll yq^{1/2+\varepsilon}(\log yq)^4.
\end{equation*}
The lemma now follows from the above estimates.
\end{proof}

\begin{lemma}
For $d\leq y$ and $(d,q)=1$, let
\begin{displaymath}
S_{j}(d)=\sum_{\substack{n\leq y/d\\(n,dq)=1}}\frac{\mu(n)}{n}(\log n)^{j}P\bigg(\frac{\log y/dn}{\log y}\bigg).
\end{displaymath}
Then $S_j(d)=M_j(d)+O(E_j(d))$ where
\begin{displaymath}
M_0(d)=\frac{dq}{\varphi(dq)\log y}P'\bigg(\frac{\log y/d}{\log y}\bigg),\ M_1(d)=-\frac{dq}{\varphi(dq)}P\bigg(\frac{\log y/d}{\log y}\bigg),\ M_j(d)=0\ (j\geq2),
\end{displaymath}
and
\begin{displaymath}
E_j(d)=(\log y)^{j-2}(\log\log y)^4\bigg(1+\frac{d^\theta\log y}{y^\theta}\bigg)\prod_{p|dq}\bigg(1+\frac{1}{p^{1-2\delta}}\bigg)^2,
\end{displaymath}
with $\theta\gg1/\log\log y$ and $\delta=1/\log\log y$.
\end{lemma}
\begin{proof}
This is Lemma 10 of Conrey \cite{C}.
\end{proof}

\begin{lemma}
Given that $f(d)=\prod_{p|d}f(p)$ with $f(p)=1+O(p^{-c})$ for $c>0$ and
\begin{displaymath}
J_j(y)={\sum_{d\leq y}{\!}}^{\star}\ \frac{\mu(d)^2}{d}f(d)\bigg(\log\frac{y}{d}\bigg)^j.
\end{displaymath}
Then we have
\begin{displaymath}
J_j(y)=\frac{1}{j+1}\prod_{p}\bigg(1-\frac{1}{p}\bigg)\bigg(1+\frac{f(p)}{p}\bigg)\prod_{p|q}\bigg(1+\frac{f(p)}{p}\bigg)^{-1}(\log y)^{j+1}+O((\log y)^j).
\end{displaymath}
\end{lemma}
\begin{proof}
This is Lemma 11 of Conrey \cite{C}.
\end{proof}

\section{Proof of Proposition \ref{prop2}}

In this section, we prove Proposition \ref{prop2}. Throughout the proof, we let $y=q^\vartheta$ and assume that $0<\vartheta<\tfrac{1}{2}$.  We begin by considering the mollified ``shifted" second moment
\begin{equation}\label{66}
J_{\alpha,\beta}(q)=%\frac{1}{\varphi^{+}(q)}
{\sum_{\chi(\textrm{mod}\ q)}{\!\!\!\!\!\!}}^{+}\ L(\tfrac{1}{2}+\alpha,\chi)L(\tfrac{1}{2}+\beta,\overline{\chi})|M(\chi)|^2,
\end{equation}
where $\alpha,\beta\in\mathbb{C}$ are small shifts satisfying $|\alpha|,|\beta| \leq (\log q)^{-1}$ and $\alpha\neq-\beta$. Applying Lemma 4.4, we have that
\begin{equation}\label{666}
\begin{split} 
J_{\alpha,\beta}(q)&= \sum_{m,n\leq y } \frac{\mu(m)\mu(n)}{\sqrt{mn}} P\Big(\frac{\log y/m}{\log y}\Big)P\Big(\frac{\log y/n}{\log y}\Big) \mathscr{B}(m,n;\alpha,\beta)
\\
&= \Sigma_1(\alpha,\beta)+\Sigma_2(\alpha,\beta) + O\big( y q^{1/2+\varepsilon}\big),
\end{split}
\end{equation}
where
$$ \Sigma_1(\alpha,\beta) =\varphi^+(q) \ \! \zeta_{q}(1+\alpha+\beta)\ {\sum_{d\leq y}{\!}}^{\star}{\sum_{\substack{m,n\leq y/d\\(m,n)=1}}{\!\!\!\!\!}}^{\star}\quad\frac{\mu(dm)\mu(dn)}{dm^{1+\beta}n^{1+\alpha}}P\bigg(\frac{\log y/dm}{\log y}\bigg)P\bigg(\frac{\log y/dn}{\log y}\bigg)$$
and 
\begin{eqnarray*}
\Sigma_2(\alpha,\beta)&=&\varphi^+(q) \ \! \bigg(\frac{q}{\pi}\bigg)^{-\alpha-\beta}g_{\alpha,\beta}^{-}(0)\zeta_{q}(1-\alpha-\beta) \times \nonumber\\
&&\qquad{\sum_{d\leq y}{\!}}^{\star}{\sum_{\substack{m,n\leq y/d\\(m,n)=1}}{\!\!\!\!\!}}^{\star}\quad\frac{\mu(dm)\mu(dn)}{dm^{1-\alpha}n^{1-\beta}}P\bigg(\frac{\log y/dm}{\log y}\bigg)P\bigg(\frac{\log y/dn}{\log y}\bigg).
\end{eqnarray*}
We can remove the restriction $(m,n)=1$  by writing $K_{\alpha,\beta}(q):=\Sigma_1(\alpha,\beta)+\Sigma_2(\alpha,\beta)$ as
\begin{equation}\label{6}
\begin{split}
&\varphi^+(q) \ {\sum_{cd\leq y}{\!}}^{\star}\ \frac{\mu(c)\mu(cd)^2}{c^2d} \ \times
\\
& \quad \quad \sum_{\substack{m,n\leq y/cd\\(mn,cdq)=1}}\frac{\mu(m)\mu(n)}{mn}P\bigg(\frac{\log y/cdm}{\log y}\bigg)P\bigg(\frac{\log y/cdn}{\log y}\bigg) Z_{q,\alpha,\beta}(m,n,c),
\end{split}
\end{equation}
where 
\begin{equation}\label{Z}
 Z_{q,\alpha,\beta}(m,n,c) = \frac{\zeta_{q}(1+\alpha+\beta)}{c^{\alpha+\beta}m^{\beta}n^{\alpha}}+\bigg(\frac{q}{\pi}\bigg)^{-\alpha-\beta}g_{\alpha,\beta}^{-}(0)\frac{\zeta_{q}(1-\alpha-\beta)}{c^{-\alpha-\beta}m^{-\alpha}n^{-\beta}}.
 \end{equation}

Though the function $\zeta_{q}(s)$ has a simple pole at $s=1$, we note that $Z_{q,\alpha,\beta}(m,n,c) $ is holomorphic in both $\alpha$ and $\beta$ in a small neighborhood of $\alpha=\beta=0$ (as can be seen, for instance, by computing the Laurent series expansion of each of the terms on the right-hand side of (\ref{Z}) about $\alpha=\beta=0$).  Therefore, the expressions in (\ref{66}) and (\ref{6}) provide an analytic continuation of the function $J_{\alpha,\beta}(q)-K_{\alpha,\beta}(q)$ to the region $|\alpha|,|\beta| \leq (\log q)^{-1}$; the function $K_{0,0}(q)$ must be defined in terms of the limit
$$ Z_{q,0,0}(m,n,c)= \lim_{\alpha\rightarrow 0} \left(\frac{\zeta_{q}(1+2\alpha)}{(c^2mn)^{\alpha}}+\bigg(\frac{q}{\pi}\bigg)^{-2\alpha} \frac{\zeta_{q}(1-2\alpha)}{(c^2mn)^{-\alpha}} \right).$$
Moreover, by the maximum modulus principle and (\ref{666}), we see that
$$ \Big| J_{\alpha,\beta}(q)-K_{\alpha,\beta}(q) \Big| \ll_\varepsilon y q^{1/2+\varepsilon}$$
uniformly for  $|\alpha|,|\beta| \leq (\log q)^{-1}$.  Hence, by Cauchy's Integral Theorem, 
\begin{equation*}
\begin{split}
\frac{d^{2k}}{d\alpha^k d\beta^k}  \Big[  J_{\alpha,\beta}(q)&-K_{\alpha,\beta}(q) \Big] \Big|_{\alpha=\beta=0} 
\\
&= \frac{(k!)^2}{(2\pi i)^2} \int_{\mathscr{C}_\alpha} \int_{\mathscr{C}_\beta} \frac{J_{w_\alpha,w_\beta}(q)-K_{w_\alpha,w_\beta}(q)}{(w_\alpha w_\beta)^{k+1}} dw_\alpha dw_\beta
\\
&\ll_{k,\varepsilon} y q^{1/2+2\varepsilon},
\end{split}
\end{equation*}
where $\mathscr{C}_\alpha$ (resp. $\mathscr{C}_\beta$) denotes the positively oriented circle in the complex plane centered at $\alpha=0$ (resp. $\beta=0$) with radius $(\log q)^{-1}.$  Thus, we have shown that
\begin{equation}\label{6666}
S_2(k,q) = \frac{d^{2k}}{d\alpha^k d\beta^k} K_{\alpha,\beta}(q) \Big|_{\alpha=\beta=0}  + \ O_{k,\varepsilon} \big(y q^{1/2+2\varepsilon}\big).
 \end{equation}
Writing
\begin{equation*}
\begin{split}
 \frac{d^{2k}}{d\alpha^k d\beta^k}& \ Z_{q,\alpha,\beta}(m,n,c)\Big|_{\alpha=\beta=0} 
 \\
 &= \sum_{h+i+j\leq 2k+1}\!\bigg(a_{h,i,j}(\log c)^{h}+b_{h,i,j}(\log q/c)^{h}\bigg)(\log m)^i(\log n)^j
 \end{split}
 \end{equation*}
for certain constants $a_{h,i,j}$ and $b_{h,i,j}$, we see that
\begin{equation}\label{dK}
\begin{split}
 &\frac{d^{2k}}{d\alpha^k d\beta^k}  K_{\alpha,\beta}(q)\Big|_{\alpha=\beta=0} 
 \\
 &\ = \varphi^+(q)\sum_{h+i+j\leq 2k+1}{\sum_{cd\leq y}{\!}}^{\star}\ \bigg(a_{h,i,j}(\log c)^{h}+b_{h,i,j}(\log cq)^{h}\bigg)\frac{\mu(c)\mu(cd)^2}{c^2d}S_i(cd)S_j(cd),
 \end{split}
 \end{equation}
where $S_i$ and $S_j$ are defined in Lemma 4.5.  It follows from Lemma 4.5 that
$$ S_i(cd) \ll_i \frac{cdq}{\varphi(cdq)} (\log y)^{i-1},$$
from which it can be seen that the contribution of the terms with $h+i+j\leq 2k$ to the sum on the right-hand side of (\ref{dK}) is 
$$\ll_k (\log q)^{2k-1}q\varphi^+(q)/\varphi(q) \ll_{k,\varepsilon} \varphi^+(q) (\log q)^{2k-1+\varepsilon}$$ since $q/\varphi(q) \ll \log\log q$.  It remains to consider the contribution of the terms with $h+i+j = 2k+1.$ In the notation of Lemma 4.5, it can be shown that
\begin{displaymath}
{\sum_{cd\leq y}{\!}}^{\star} \ \frac{S_i(cd)E_j(cd)}{c^2d}\ll_{i,j,\varepsilon} (\log y)^{i+j-2+\varepsilon}
\end{displaymath}
and
\begin{displaymath}
{\sum_{cd\leq y}{\!}}^{\star} \ \frac{E_i(cd)E_j(cd)}{c^2d}\ll_{i,j,\varepsilon}  (\log y)^{i+j-3+\varepsilon}.
\end{displaymath}
Hence the contribution of the error terms $E_i$ and $E_j$, arising from Lemma 4.5, to the terms in (\ref{dK}) with $h+i+j = 2k+1$ is $\ll_{k,\varepsilon} \varphi^+(q)(\log q)^{2k-1+\varepsilon}$.  Thus,
\begin{equation*}\label{dK2}
\begin{split}
 & \frac{d^{2k}}{d\alpha^k d\beta^k}  K_{\alpha,\beta}(q)\Big|_{\alpha=\beta=0} 
 \\
 & \quad = \varphi^+(q)\!\!\!\!\!\sum_{h+i+j = 2k+1}{\sum_{cd\leq y}{\!}}^{\star}\ \bigg(a_{h,i,j}(\log c)^{h}+b_{h,i,j}(\log cq)^{h}\bigg)\frac{\mu(c)\mu(cd)^2}{c^2d}M_i(cd)M_j(cd)
 \\
 &\qquad\quad + O_{k,\varepsilon} \big(\varphi^+(q) (\log q)^{2k-1+\varepsilon}\big).
 \end{split}
 \end{equation*}
Since $M_i(cd)=0$ for $i>1$, we need only to consider the terms with $0\leq i,j \leq 1$.  Moreover, the terms involving powers of $\log c$ can be ignored, as they contribute (due to the presence of $c^{-2}$ in the sum) an amount which is $\ll_{k,\varepsilon} (\log q)^{2k-1+\varepsilon}$.  Therefore, the above expression simplifies to 
\begin{equation}\label{dK3}
\begin{split}
 \frac{d^{2k}}{d\alpha^k d\beta^k} & K_{\alpha,\beta}(q)\Big|_{\alpha=\beta=0} = T_1 + 2 T_2 + T_3 + O_{k,\varepsilon} \big( \varphi^+(q)(\log q)^{2k-1+\varepsilon}\big),
 \end{split}
 \end{equation}
where 
$$ T_1= \varphi^+(q){\sum_{cd\leq y}{\!}}^{\star}\ b_{2k+1,0,0}(\log q)^{2k+1} \frac{\mu(c)\mu(cd)^2}{c^2d}M_0(cd)^2$$
$$ T_2=\varphi^+(q){\sum_{cd\leq y}{\!}}^{\star}\ b_{2k,1,0}(\log q)^{2k} \frac{\mu(c)\mu(cd)^2}{c^2d}M_0(cd)M_1(cd)$$
and 
$$ T_3 = \varphi^+(q){\sum_{cd\leq y}{\!}}^{\star}\ b_{2k-1,1,1}(\log q)^{2k-1} \frac{\mu(c)\mu(cd)^2}{c^2d}M_1(cd)^2.$$

We first evaluate $T_1$. Using Lemma 4.5 we have that
\begin{eqnarray*}
T_1&=&\varphi^+(q)\frac{b_{2k+1,0,0}q^2(\log q)^{2k+1}}{\varphi(q)^2(\log y)^2}{\sum_{cd\leq y}{\!}}^{\star}\ \frac{\mu(c)\mu(cd)^2d}{\varphi(cd)^2}P'\bigg(\frac{\log y/cd}{\log y}\bigg)^2\nonumber\\
&=&\varphi^+(q)\frac{b_{2k+1,0,0}q^2(\log q)^{2k+1}}{\varphi(q)^2(\log y)^2}{\sum_{n\leq y}{\!}}^{\star}\ \frac{\mu(n)^2}{\varphi(n)}P'\bigg(\frac{\log y/n}{\log y}\bigg)^2.
\end{eqnarray*}
Now Lemma 4.6 implies that  
\begin{displaymath}
{\sum_{n\leq y}{\!}}^{\star}\ \frac{\mu(n)^2}{\varphi(n)}P'\bigg(\frac{\log y/n}{\log y}\bigg)^2=\frac{\varphi(q)}{q}\big(\log y+O(1)\big)\int_{0}^{1}P'(x)^2dx.
\end{displaymath}
Hence
\begin{equation}\label{T1}
T_1=\varphi^+(q)\frac{b_{2k+1,0,0}q(\log q)^{2k+1}}{\varphi(q)\log y}\int_{0}^{1}P'(x)^2dx+O_{k,\varepsilon}(\varphi^+(q)(\log q)^{2k-1+\varepsilon}).
\end{equation}
Similarly, it can be shown that
\begin{eqnarray}\label{T2}
T_2&=&-\varphi^+(q)\frac{b_{2k,1,0}q(\log q)^{2k}}{\varphi(q)}\int_{0}^{1}P'(x)P(x)dx+O_{k,\varepsilon}(\varphi^+(q)(\log q)^{2k-1+\varepsilon})\nonumber\\
&=&-\frac{b_{2k,1,0}q(\log q)^{2k}}{2\varphi(q)}+O_{k,\varepsilon}(\varphi^+(q)(\log q)^{2k-1+\varepsilon}),
\end{eqnarray}
and that
\begin{equation}\label{T3}
T_3=\varphi^+(q)\frac{b_{2k-1,1,1}q(\log q)^{2k-1}\log y}{\varphi(q)}\int_{0}^{1}P(x)^2dx+O_{k,\varepsilon}(\varphi^+(q)(\log q)^{2k-1+\varepsilon}).
\end{equation}
Thus, combining (\ref{6666}), (\ref{dK3}), (\ref{T1}), (\ref{T2}), (\ref{T3}) and noting that
\begin{displaymath}
b_{2k+1,0,0}=\frac{\varphi(q)}{q(2k+1)},\quad b_{2k,0,1}=-\frac{\varphi(q)}{2q},\quad \textrm{and} \quad b_{2k-1,1,1}=\frac{\varphi(q)k^2}{q(2k-1)},
\end{displaymath}
it follows that, for $y=q^\vartheta$ and $0<\vartheta<\tfrac{1}{2},$ 
\begin{eqnarray*}
S_2(k,q) &=&\bigg(\frac{\vartheta^{-1}}{2k+1}\int_{0}^{1}P'(x)^2dx+\frac{1}{2}+\frac{\vartheta k^2}{2k-1}\int_{0}^{1}P(x)^2dx\bigg)\varphi^+(q)(\log q)^{2k}\nonumber\\
&&\qquad\qquad+O_{k,\varepsilon}(\varphi^+(q)(\log q)^{2k-1+\varepsilon}).
\end{eqnarray*}
This completes the proof of Proposition \ref{prop2}.

\section{Completing the Proof of Theorem \ref{th2}: Optimizing the mollifier}

We are now in a position to complete the proof of Theorem 1.  By Proposition \ref{prop1} and Propostion \ref{prop2}, for $0<\vartheta<\tfrac{1}{2}$, we see that
\begin{equation}\label{tt}
P_{k}^* \geq \left[ \frac{\vartheta^{-1}}{2k\!+\!1} \int_{0}^{1} P'(x)^{2} \ dx +\frac{1}{2} + \frac{\vartheta k^{2}}{2k\!-\!1}  \int_{0}^{1} P(x)^{2} \ dx  \right]^{-1}.
\end{equation}
For each choice of $k\in\mathbb{N}$, we wish to find a polynomial $P$ satisfying $P(0)=1$ and $P(1)=0$ that maximizes the expression on the right-hand side of the above inequality. Equivalently, we wish to minimize the expression 
\begin{equation}\label{ttt}
F_{k}(P) := \frac{\vartheta^{-1}}{2k\!+\!1} \int_{0}^{1} P'(x)^{2} \ dx  + \frac{\vartheta k^{2}}{2k\!-\!1}  \int_{0}^{1} P(x)^{2} \ dx.
\end{equation}
This optimization problem is solved explicitly in \textsection $7$ of \cite{MV} (and, independently, in \cite{CG}; see the remarks on page 97). We recall the argument given by Michel and Vanderkam in \cite{MV}.

Using a standard approximation argument, the polynomial $P$ can be replaced by any infinitely differentiable function with a rapidly convergent Taylor series on $[0,1]$. In this case, using the calculus of variations, the optimization problem can be explicitly solved and, for $k>0$, the optimal choice of $P$ is $$ P(t) = \frac{\sinh(\Lambda t)}{\sinh(\Lambda)}, \quad \text{ where } \  \Lambda = \vartheta k \sqrt{\frac{2k\!+\!1}{2k\!-\!1}}.$$ With this choice of $P$, it follows that
\begin{equation}\label{FK}
F_{k}(P) = \frac{\Lambda \coth \Lambda}{\vartheta(2k\!+\!1)} = \frac{k \coth\Lambda}{\sqrt{4k^{2}\!-\!1}}.
\end{equation}
As $k$ gets large, the function $\coth\Lambda \rightarrow 1$ and so asymptotically (as $k\rightarrow\infty$)
$$F_{k}(P) = \frac{1}{2}+\frac{1}{16 k^{2}} +O\Big(\frac{1}{k^{4}}\Big).$$
When combined with (\ref{tt}) and (\ref{ttt}), this asymptotic formula is enough to establish the estimate for $P_{k}^{*}$ in (\ref{pro}) and, thus, completes the proof of Theorem \ref{th2}.

\begin{table}[ht]
\caption{In the table below, lower bounds for the proportions $P_{k}$ and $P_{k}^{*}$, defined in equations (\ref{PK}) and (\ref{PstarK}), respectively. These calculations were performed by using the expression for $F_k(P)$ given in (\ref{FK}) with $\vartheta=\frac{1}{2} - 1\times 10^{\text{-}8}$.}\label{eqtable}
\renewcommand\arraystretch{1.5}
\noindent\[
\begin{array}{|c|c|c|}
\hline
\ k \ &\text{ Lower bound for }P_{k} \ &\text{ Lower bound for }P_{k}^{*}  \ \\
\hline
$1$ & { \tfrac{2}{3} \times}0.8216\dots &0.7544\dots \\
$2$ &  { \tfrac{2}{3} \times}0.9369\dots &0.9083\dots \\
$3$ & { \tfrac{2}{3} \times}0.9758\dots &0.9642\dots \\
$4$ & { \tfrac{2}{3} \times}0.9901\dots &0.9853\dots \\
$5$ & { \tfrac{2}{3} \times}0.9956\dots &0.9935\dots \\
$10$ & { \tfrac{2}{3} \times}0.9995\dots &0.9993\dots \\
$15$ & { \tfrac{2}{3} \times}0.9997\dots &0.9997\dots \\
$20$ & { \tfrac{2}{3} \times}0.9998\dots &0.9998\dots \\
$25$ & { \tfrac{2}{3} \times}0.9999\dots &0.9999\dots \\
\hline
\end{array}
\]
\end{table}

\section*{Acknowledgments}

Much of this research was completed while the first author was visiting the University of Rochester and the University of Mississippi. He would like to thank these institutions for their hospitality. Both authors would like to thank Professor K. Soundararajan for pointing out an error in the original version of this manuscript and to thank Professor Steven M. Gonek for his support and encouragement.

\end{document}